\documentclass[12pt]{article}
\usepackage{amsmath,amssymb}
\usepackage[T1]{fontenc}
\usepackage{latexsym}
\usepackage{pdfsync}
\usepackage{epsfig}
\usepackage{proof}
\usepackage[utf8]{inputenc}
\usepackage[a4paper, top=3cm, bottom=3cm ]{geometry}
\usepackage{graphicx}

\newtheorem{thm}{Theorem}[section]
\newtheorem{defn}{Definition}[section]
\newtheorem{lem}{Lemma}[section]

\newtheorem{rem}{Remark}[section]

\title
{$L^p$-norm estimate for the Bergman projection on  Hartogs triangle}

\author{\normalsize Tomasz Beberok \\
\small Faculty of Mathematics and Computer Science, Jagiellonian University,\\
\small Lojasiewicza 6, 30-048 Krakow, Poland \\}

\date{}

\begin{document}

\begin{center}
  \textbf{Markov's inequality on Koornwinder's domain in $L^p$ norms
}
\end{center}
\vskip1em
\begin{center}
  Tomasz Beberok
\end{center}

\vskip2em

\noindent \textbf{Abstract.} Let $\Omega=\{(x,y) \in \mathbb{R}^2 \colon |x|<y+1, \, x^2>4y\}$. We prove that the optimal exponent in Markov's inequality on $\Omega$ in $L^p$ norms is 4.
\vskip1em


\section{Introduction}
\label{}
Throughout this paper $\mathcal{P}(\mathbb{R}^N)$ ($\mathcal{P}_n(\mathbb{R}^N)$, respectively) denotes the set of algebraic polynomials of $N$ variables with real coefficients (with total degree at most $n$). We begin with the definition of multivariate Markov's inequality.
\begin{defn}
Let $E \subset \mathbb{R}^N$ be a compact set. We say that $E$ admit Markov's inequality if there exist constants $M,r >0$ such that for every polynomial $P \in \mathcal{P}(\mathbb{R}^N)$ and $j \in \{1,2,\ldots,N\}$
 \begin{align}\label{Markov}
    \left\| \frac{\partial P}{\partial x_j} \right\|_E \leq M (\deg P)^r \|P\|_E
 \end{align}
where $\|\cdot\|_E$ is the supremum norm on $E$.
\end{defn}
A compact set $E$ with this property is called a Markov set. The inequality (\ref{Markov}) is a generalization of the classical inequality proved by A. A. Markov in 1889, which gives such estimate on $[-1,1]$. The theory of Markov inequality and it's generalizations is still the active and fruitful area of approximation theory (see, for instance, \cite{LCK,B,KN}). For a given compact set $E$, an important problem is to determine the minimal constant $r$ in (\ref{Markov}). This can be used to minimize the loss of regularity in problems concerning the linear extension of classes of $\mathrm{C}^{\infty}$ functions with restricted growth of derivatives (see \cite{PP,P1}). Such an $r$ is so-called Markov's exponent of $E$ (see \cite{BP} for more detail on this matter). In the case of supremum norm various information about Markov's exponent is known (see, e.g., \cite{DL,G1,MM,P2,RS}). The Markov type inequalities were also studied in $L^p$ norms (see \cite{MB,BE,G11,G2,G3,HST}). In this case the question of Markov's exponent problem is much more complex. In particular,  to the best of our knowledge, there is no example of a compact set in $\mathbb{R}^N$ with cusps for which Markov's exponent (with respect to the Lebesgue measure) is known. The attempts to solve this problem led, among others, to a so-called Milówka--Ozorka identity (see \cite{BKMO,M1,M2} for discussion). The aim of this note is to give such an example. More precisely, we show that, in the notation above,  Markov's exponent of $\Omega$ in $L^p$ norms is 4. Here $\Omega=\{(x,y) \in \mathbb{R}^2 \colon |x|<y+1, \, x^2>4y\}$ which is depicted in the Figure \ref{Koorn}. Since $\Omega$ is the region for Koornwinder orthogonal polynomials (first type), see \cite{Koo1,Koo2}, we call this set  Koornwinder's domain.

\begin{figure}[ht]\label{Koorn}
\centering
\includegraphics[scale=0.5]{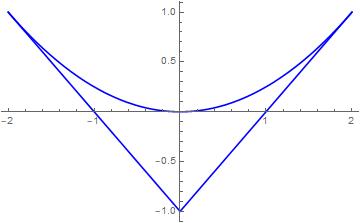}
\caption{Region for Koornwinder orthogonal polynomials.}
\end{figure}

\section{Some weighted polynomial inequalities on simplex}
The following lemma will be particularly useful in the proof of our main result.
\begin{lem}\label{WMI}
  Let $S=\{(x_1,x_2) \in \mathbb{R}^2 \colon -1<x_1<x_2<1\}$, $w(x_1,x_2)=x_2-x_1$ and $1 \leq p \leq \infty$. Then there exists a positive constant $C(S,w)$ such that, for every $P \in \mathcal{P}_n(\mathbb{R}^2)$, we have
     \begin{align}\label{wMI}
       \left\|\frac{\partial P}{\partial x_i}\right\|_{L^p(S,w)} \leq C n^2 \|P\|_{L^p(S,w)} \quad (i=1,2).
     \end{align}
\end{lem}
\begin{proof}
We start with  $p=\infty$. Since $\overline{S}$ is a convex body in $\mathbb{R}^2$, the result of Wilhelmsen \cite{W} gives
   \begin{align*}
     \max\left\{\left\|w\frac{\partial P}{\partial x_1}-P\right\|_{L^{\infty}(S)},  \left\|w\frac{\partial P}{\partial x_2}+P\right\|_{L^{\infty}(S)} \right\} \leq \frac{2(n+1)^2}{r_{S}} \|wP\|_{L^{\infty}(S)},
   \end{align*}
where $\delta_S$ is the width of the convex body (the minimal distance between parallel supporting hyperplanes). Therefore by Lemma 3 from \cite{G2}, there is a constant $\kappa > 0$ such that, for all $P \in \mathcal{P}_n(\mathbb{R}^2)$,
   \begin{align}\label{infty}
     \left\|w\frac{\partial P}{\partial x_i}\right\|_{L^{\infty}(S)} \leq \frac{2(\kappa \delta_{S}+1)(n+1)^2}{\delta_{S}} {\|P\|}_{L^{\infty}(S,w)} \quad (i=1,2).
   \end{align}
Thus we conclude that (\ref{wMI}) holds when $p=\infty$. Now, for each  $1 \leq p < \infty$, it is clear that
    \begin{align*}
      \left\|\frac{\partial P}{\partial x_i}\right\|_{L^p(S,w)} \leq \sum_{j=0}^{2} \left( \int_{D_j} \left| \frac{\partial P}{\partial x_i} (x_1,x_2) \right|^p (x_2-x_1) \, dx_1dx_2 \right)^{1/p}
    \end{align*}
where
   \begin{align*}
     &D_0=\{(x_1,x_2) \in \mathbb{R}^2 \colon -1<x_1<0, \, x_1+1<x_2<1\}, \\ &D_1=\{(x_1,x_2) \in \mathbb{R}^2 \colon -1<x_1<0, \, x_1<x_2<x_1+1\}, \\ &D_2=\{(x_1,x_2) \in \mathbb{R}^2 \colon 0<x_2<1, \, x_2-1<x_1<x_2\}.
   \end{align*}
We shall show that there is a constant $\tilde{C} > 0$ such that, for all $P \in \mathcal{P}(\mathbb{R}^2)$,
   \begin{align}\label{parts}
      \left\|\frac{\partial P}{\partial x_i}\right\|_{L^p(D_j,w)} \leq \tilde{C} (\deg P)^{2} {\|P\|}_{L^p(S,w)}, \quad j=0,1,2.
   \end{align}
Since $D_0$ is a bounded convex set and $w$ is bounded away from zero on $D_0$, we have (see \cite{D,G11,G2,Kroo})
    \begin{align*}
      \left( \int_{D_0} \left| \frac{\partial P}{\partial x_i} (x_1,x_2) \right|^p (x_2-x_1) \, dx_1dx_2 \right)^{1/p} &\leq  C_0 (\deg P)^{2} \|P\|_{L^p(D_0,w)}\\ &\leq C_0 (\deg P)^{2} \|P\|_{L^p(S,w)}.
   \end{align*}
Now consider the case $j=1$. The integral is then
    \begin{align*}
      \left( \int_{D_1} \left| \frac{\partial P}{\partial x_i} (x_1,x_2) \right|^p (x_2-x_1) \, dx_1dx_2 \right)^{1/p}.
    \end{align*}
We perform the change of variables $t=x_1$, $s=x_2-x_1$. The integral becomes
     \begin{align*}
      \left( \int_{-1}^{0} \int_{0}^{1} \left| \frac{\partial P}{\partial x_i} (t,s+t) \right|^p s \, dsdt \right)^{1/p}.
    \end{align*}
Define $Q(t,s)=P(t,s+t)$. Then
    \begin{align*}
    \frac{\partial Q}{\partial t} (t,s) - \frac{\partial Q}{\partial s} (t,s) = \frac{\partial P}{\partial x_1} (t,s+t), \quad \frac{\partial Q}{\partial s} (t,s) = \frac{\partial P}{\partial x_2} (t,s+t).
    \end{align*}
Hence, (using Goetgheluck’s result--see \cite{G0})
    \begin{align*}
      \int_{0}^{1} \left| \frac{\partial P}{\partial x_2} (t,s+t) \right|^p s \, ds \leq C_1^p (\deg Q)^{2p} \int_{0}^{1} \left| Q(t,s) \right|^p s \, ds.
    \end{align*}
Therefore
    \begin{align*}
       \left\|\frac{\partial P}{\partial x_2}\right\|_{L^p(D_1,w)}  & \leq \left( \int_{-1}^{0} \left[ C_1^p (\deg P)^{2p}  \int_{0}^{1} \left| Q(t,s) \right|^p s \, ds \right] dt
       \right)^{1/p} \\&= C_1 (\deg P)^{2}  \left( \int_{-1}^{0}   \int_{0}^{1} \left| P(t,s+t) \right|^p s \, ds dt\right)^{1/p} \\ &\leq C_1 (\deg P)^{2} \|P\|_{L^p(S,w)}.
    \end{align*}
On the other hand,
   \begin{align*}
     \left\|\frac{\partial P}{\partial x_1}\right\|_{L^p(D_1,w)} \leq & \left( \int_{-1}^{0} \int_{0}^{1} \left| \frac{\partial Q}{\partial t} (t,s) \right|^p s \, dsdt \right)^{1/p} \\ &+ \left( \int_{-1}^{0} \int_{0}^{1} \left| \frac{\partial Q}{\partial s} (t,s) \right|^p s \, dsdt \right)^{1/p}.
   \end{align*}
We have, arguing as before, that there exists constants $\hat{C}_1, C_1$ such that for every polynomial $Q \in \mathcal{P}(\mathbb{R}^2)$
   \begin{align*}
      &  \int_{-1}^{0} \left| \frac{\partial Q}{\partial t} (t,s) \right|^p  \, dt \leq \hat{C}_1^p (\deg Q)^{2p} \int_{-1}^{0} \left| Q(t,s) \right|^p  \, dt,\\
      &  \int_{0}^{1} \left| \frac{\partial Q}{\partial s} (t,s) \right|^p s \, ds \leq C_1^p (\deg Q)^{2p} \int_{0}^{1} \left| Q(t,s) \right|^p s \, ds .
   \end{align*}
Therefore we see immediately that
   \begin{align*}
    \left\|\frac{\partial P}{\partial x_1}\right\|_{L^p(D_1,w)}   \leq & \left( \int_{0}^{1} \left[ \hat{C}_1^p (\deg P)^{2p} s \int_{-1}^{0} \left| Q(t,s) \right|^p  \, dt \right] ds \right)^{1/p} \\ &+ \left( \int_{-1}^{0} \left[ C_1^p (\deg P)^{2p}  \int_{0}^{1} \left| Q(t,s) \right|^p s \, ds \right] dt\right)^{1/p}.
   \end{align*}
Thus we finally have
   \begin{align*}
     \left\|\frac{\partial P}{\partial x_1}\right\|_{L^p(D_1,w)}  &\leq \hat{C}_1 (\deg P)^{2} \|P\|_{L^p(D_1,w)} + C_1 (\deg P)^{2} \|P\|_{L^p(D_1,w)} \\ &\leq (\hat{C}_1  + C_1) (\deg P)^{2} \|P\|_{L^p(S,w)}.
   \end{align*}
A similar result for $D_2$ obtains if one considers the substitution $t=x_2$, $s=x_2-x_1$ and polynomial $\tilde{Q}(t,s)=P(t-s,t)$. We omit the details. Thus we have shown that, if $\tilde{C}=2 \max\{C_0,\hat{C}_1,C_1,\hat{C}_2,C_2 \}$, then (\ref{parts}) holds. That completes the proof.
\end{proof}
Now we shall prove the following weighted Schur-type inequality.
\begin{lem}\label{WSchur}
  (with previous notation). Let $d$ be a natural number. Then, for every $A \subset \bar{S}$ and $R \in \mathcal{P}_k(\mathbb{R}^2)$, satisfying the condition
   \begin{align*}
     \{ \text{there exists }\, \alpha \in \mathbb{N}^2 \, \text{ such that } \, \alpha_1 + \alpha_2 \leq d \, \text{ and } \, |R^{(\alpha)}(x)| \geq m > 0 \,\, (x \in A) \}
   \end{align*}
  one can find a constant $C_d$ such that, for any $\epsilon >0$ and every $P \in \mathcal{P}_n(\mathbb{R}^2)$, we have
    \begin{align}\label{wSchur}
      \|P\|_{L^p(A,w)} \leq C_d m^{-1} \epsilon^{-1} (n+k)^{2d} \|PR\|_{L^p(S,w)} + \epsilon \|P\|_{L^p(S,w)}.
    \end{align}
\end{lem}
\begin{proof}
  The idea of the proof comes from \cite{G2}. Thus we proceed by induction on the length of $\alpha$. If $\alpha_1=\alpha_2=0$, then
    \begin{align*}
      |P(x)| \leq m^{-1}|P(x)R(x)| \quad \text{for} \quad x \in A.
    \end{align*}
  Therefore
     \begin{align*}
       \|P\|_{L^p(A,w)} \leq m^{-1} \|PR\|_{L^p(A,w)} \leq m^{-1}  \|PR\|_{L^p(S,w)} + \epsilon \|P\|_{L^p(S,w)}.
     \end{align*}
  Now, for $d_0 \leq d$, assume that (\ref{wSchur}) holds when $\alpha_1 + \alpha_2 < d_0$. We shall show that
  (\ref{wSchur})  is still valid for $\alpha$ such that $|\alpha|=d_0$. Here $|\alpha|$  denotes the  length of $\alpha$. Let
     \begin{align*}
       I=\left\{(\beta_1,\beta_2) \in \mathbb{N}^2 \colon 0< |\beta|, \, 0 \leq \beta_1 \leq \alpha_1, \,  0 \leq \beta_2 \leq \alpha_2\right\}.
     \end{align*}
  Notice that the set $I$  contains at most $\frac{(d_0+1)(d_0+2)}{2} -1$ elements. By Leibniz’s rule, if $x \in A$, then
     \begin{align*}
       |P(x)| \leq m^{-1} \left[ |(PR)^{(\alpha)}(x)| + \sum_{\beta \in I} \binom{\alpha}{\beta} |R^{(\alpha - \beta)}(x)|\,  |P^{(\beta)}(x)| \right].
     \end{align*}
  Let $C$ be a constant so that (\ref{wMI}) holds. We set 
     \begin{align*}
      B_0=\{ x \in A : |R^{(\alpha - \beta)}(x)| \leq  \frac{m \epsilon}{ \eta^{2}} {{\alpha}\choose{\beta}}^{-1}  (C n^2)^{-|\beta|}, \, \beta \in I \},
     \end{align*}
  where $\eta=\frac{(d_0+1)(d_0+2)}{2}$. Then, for each $x\in B_0$, we have
     \begin{align*}
       |P(x)| \leq m^{-1} |(PR)^{(\alpha)}(x)| +  \frac{ \epsilon}{ \eta^{2}} \sum_{\beta \in I} (Cn^2)^{-|\beta|} |P^{(\beta)}(x)|.
     \end{align*}
  This yields
  \begin{align*}
     \|P\|_{L(B_0,w)} &\leq  m^{-1} \|(PR)^{(\alpha)}\|_{L(B_0,w)} + \frac{ \epsilon}{\eta^2} \sum_{\beta \in I} (Cn^2)^{-|\beta|} \|P^{(\beta)}\|_{L(B_0,w)} \\ &\leq m^{-1} \|(PR)^{(\alpha)}\|_{L(S,w)} + \frac{ \epsilon}{\eta^2} \sum_{\beta \in I} (Cn^2)^{-|\beta|} \|P^{(\beta)}\|_{L(S,w)}.
  \end{align*}
  Therefore by the preceding lemma,
    \begin{align*}
      {\|P\|}_{L^p(B_0,w)} \leq  m^{-1} C^{|\alpha|} (n+k)^{2|\alpha|} {\|PR\|}_{L^p(S,w)} + \frac{\epsilon}{\eta} {\|P\|}_{L^p(S,w)}.
    \end{align*}
  On the other hand, if $x \in A\setminus B_0$, then there exists $\beta \in I$ such that
    \begin{align}\label{rineq}
      |R^{(\alpha - \beta)}(x)| > {{\alpha}\choose{\beta}}^{-1}  \frac{ m \epsilon (Cn^2)^{-|\beta|} }{\eta^2}.
    \end{align}
  Thus we can share $x \in A\setminus B_0$  into at most $\eta-1$ disjoint subsets $B_j$ such that, for every $x \in B_j$, there exists an index $\beta$ for which (\ref{rineq}) holds. Therefore, since $|\beta|>0$, on each $B_j$, replacing $\epsilon$ by $\frac{\epsilon}{\eta}$, we conclude by induction that
    \begin{align*}
      \|P\|_{L(B_j,w)} \leq (Cn^2)^{|\beta|} & \frac{\eta^2}{ m \epsilon}  {\binom{\alpha}{\beta}} C_{d_0-1} \left(\frac{\eta}{ \epsilon}\right)^{d_0-1} \\ \times & (n+k)^{2(d_0-|\beta|)} \|PR\|_{L(S,w)} + \frac{\epsilon}{\eta} \|P\|_{L(S,w)}.
    \end{align*}
  Since $A=\bigcup_{j} B_j$ we see that
    \begin{align*}
       \|P\|_{L^p(A,w)} \leq C_{d_0} m^{-1} \epsilon^{-d_0} (n+k)^{2d_0} \|PR\|_{L^p(S,w)} + \epsilon \|P\|_{L^p(S,w)}
    \end{align*}
  with
     \begin{align*}
       C_{d_0}= C^{2d_0} + \left(\frac{(d_0+1)(d_0+2)}{2 }\right)^{d_0+1} C_{d_0-1} \sum_{\beta \in I} \binom{\alpha}{\beta} C^{|\beta|},
     \end{align*}
  which completes the induction and the proof.
\end{proof}
\section{Main result}
Our main result reads as follows:
   \begin{thm}
     Let $p \geq 1$. Then there exists constant $C=C(\Omega,p)$ such that for every polynomial $P \in \mathcal{P}_n(\mathbb{R}^2)$ we have
         \begin{align}
           \max\left\{  \left\| \frac{\partial P}{\partial x} \right\|_{L^p(\Omega)}, \left\| \frac{\partial P}{\partial y} \right\|_{L^p(\Omega)} \right\} \leq C n^4  \left\| P \right\|_{L^p(\Omega)} . \label{n=4}
         \end{align}
   \end{thm}
   \begin{proof}
   Let us first prove the inequality (\ref{n=4}) with respect to the second variable. Let $P \in \mathcal{P}_n(\mathbb{R}^2)$. Then the integrals
     \begin{align*}
       \int_{\Omega} \left| \frac{\partial P}{\partial y} (x,y) \right|^p \, dx dy, \quad \int_{\Omega} \left| P (x,y) \right|^p \, dx dy
     \end{align*}
   become, under a change of variables $x = u + v$, $y = uv$,
     \begin{align*}
       \int_{S} \left| \frac{\partial P}{\partial y} (u+v,uv) \right|^p (v-u) \, du dv, \quad \int_{S} \left| P (u+v,uv) \right|^p (v-u) \, du dv
     \end{align*}
   where $S=\{(u,v) \in \mathbb{R}^2 \colon -1<u<v<1\}$. Let us define polynomial $Q(u,v)=P(u+v,uv)$. Then
     \begin{align*}
       (v-u) \frac{\partial P}{\partial y} (u+v,uv)  =  \frac{\partial Q}{\partial u} (u,v) -  \frac{\partial Q}{\partial v} (u,v).
     \end{align*}
   We now see, using Lemma \ref{WMI}, that
     \begin{align*}
      \left\| (v-u) \frac{\partial P}{\partial y} (u+v,uv)  \right\|_{L^p(S,w)} = \left\|\frac{\partial Q}{\partial u} -  \frac{\partial Q}{\partial v} \right\|_{L^p(S,w)} \leq C(p,S) (2n)^2  \left\| Q \right\|_{L^p(S,w)}.
     \end{align*}
   Lemma \ref{WSchur} tells us that
     \begin{align*}
       \left\|  \frac{\partial P}{\partial y} (u+v,uv)  \right\|_{L^p(S,w)}   \leq C_1(p,S) (2n)^2  \left\| (v-u) \frac{\partial P}{\partial y} (u+v,uv)  \right\|_{L^p(S,w)}.
     \end{align*}
   Hence
     \begin{align*}
       \left\|  \frac{\partial P}{\partial y} (u+v,uv)  \right\|_{L^p(S,w)}   \leq  C C_1 (2n)^4  \left\| Q \right\|_{L^p(S,w)}.
     \end{align*}
   That completes the proof of (\ref{n=4}) for the derivative of $P$ with respect to $y$.  To prove the remaining part we need to consider the polynomials $uQ$ and $vQ$. Then
      \begin{align*}
         (v-u) \frac{\partial P}{\partial x} (u+v,uv)= \frac{\partial vQ}{\partial v} (u,v) -  \frac{\partial uQ}{\partial u} (u,v).
      \end{align*}
   Hence
      \begin{align*}
      \left\| (v-u) \frac{\partial P}{\partial x} (u+v,uv)  \right\|_{L^p(S,w)}  &\leq C (2n+1)^2 \left( \left\| vQ \right\|_{L^p(S,w)} + \left\| uQ \right\|_{L^p(S,w)}\right) \\ &\leq C' (2n+1)^2   \left\| Q \right\|_{L^p(S,w)}.
     \end{align*}
   Thus using an argument similar to the one that we carry out in detail in the previous case, one can obtain the desired estimate.
   \end{proof}
\begin{rem}
  In the same fashion, we may prove that there exists a positive constant $C_l$ such that for every $P \in \mathcal{P}_n(\mathbb{R}^2)$ we have
     \begin{align}
       \max\left\{  \left\| \frac{\partial P}{\partial x} \right\|_{L^p(\Delta_l)}, \left\| \frac{\partial P}{\partial y} \right\|_{L^p(\Delta_l)} \right\} \leq C_l n^{2l}  \left\| P \right\|_{L^p(\Delta_l)} \label{2l}
     \end{align}
where $\Delta_l=\left\{ (x,y) \in \mathbb{R}^2 \colon |x|^{1/l} + |y|^{1/l} \leq 1\right\}$ and $l$ is a positive odd number.
\end{rem}
\section{Sharpness of the exponents}
In fact, according to \cite{BK}, it is enough to prove sharpness in the supremum norm.
The discussion here is based on unpublished work of M. Baran.  Let us consider following sequence of polynomials
    \begin{align*}
      P_k(x,y)=\left[ \frac{1}{k} T_k' \left(  \frac{2-x}{4} \right) \right]^5 \left( \frac{1+x+y}{4}\right)
    \end{align*}
where $T_k$ is the $k$th Chebyshev polynomial of the first kind. Note that the degree of a polynomial $P_k$ is equal $5k-4$. Since
    \begin{align*}
      \left|  \frac{1}{k} T_k' (1-x) \right| \leq \frac{1}{\sqrt{x}} \quad \text{for every} \quad x \in (0,1] \quad \text{and} \\ \frac{1+x+y}{4} \leq \left( \frac{1}{2} + \frac{x}{4} \right)^2 \quad \text{for} \quad (x,y) \in \Omega,
    \end{align*}
we may conclude that
    \begin{align*}
      \left| P_k(x,y) \right| \leq \left| \frac{1}{k} T_k' \left(  \frac{2-x}{4} \right) \sqrt{\frac{1}{2} + \frac{x}{4}} \right|^4 \left| \frac{1}{k} T_k' \left(  \frac{2-x}{4} \right) \right| \leq \frac{1}{k} \|T_k'\|_{[-1,1]}=k
    \end{align*}
for any $(x,y) \in \Omega$. On the other hand,
    \begin{align*}
      \left| \frac{\partial P_k}{\partial y} (-2,1) \right| = \frac{1}{4} \left| \frac{1}{k} T_k'(1) \right|^5=\frac{k^5}{4} \geq \frac{k^4}{4} \left\|P_k\right\|_{\Omega}.
    \end{align*}
A similar calculation shows that, for $Q_k=\left[ \frac{1}{k} T_k' \left(  \frac{1+y}{2}  \right) \right]^5 \left( \frac{x^2}{4} -y \right) $,
    \begin{align*}
      \left\| Q_k \right\|_{\Omega} \leq k \quad and \quad \left| \frac{\partial Q_k}{\partial x} (2,1) \right|=k^5.
    \end{align*}

Let $P_n^{(\alpha,\beta)}$ denote the Jacobi polynomials. In order to prove sharpness of (\ref{2l}), we consider the sequence of polynomials $W_n(x,y)=yP_n^{(\alpha,\alpha)}(x)$. Thus
 \begin{align*}
    &\int_{\Delta_l} \left| \frac{\partial W_n}{\partial y} (x,y) \right|^p \, dx dy = 2\int_{-1}^{1} \left| P_n^{(\alpha,\alpha)} (x) \right|^p \left(1-|x|^{1/l}\right)^l\, dx, \\
    &\int_{\Delta_l} \left| W_n (x,y) \right|^p \, dx dy = \frac{2}{p+1} \int_{-1}^{1} \left| P_n^{(\alpha,\alpha)} (x) \right|^p \left(1-|x|^{1/l}\right)^{(p+1)l}\, dx.
 \end{align*}
By the well known symmetry relation (see \cite{Szego}, Chap. IV)
  \begin{align*}
    P_n^{(\alpha,\beta)}(x)=(-1)^n P_n^{(\beta,\alpha)}(-x).
  \end{align*}
we find that
    \begin{align*}
    &\int_{\Delta_l} \left| \frac{\partial W_n}{\partial y} (x,y) \right|^p \, dx dy = 4\int_{0}^{1} \left|  P_n^{(\alpha,\alpha)}(x) \right|^p \left(1-x^{1/l}\right)^l \, dx, \\
    &\int_{\Delta_l} \left|  W_n (x,y) \right|^p \, dx dy  = \frac{4}{p+1}\int_{0}^{1} \left|  P_n^{(\alpha,\alpha)}(x) \right|^p \left(1-x^{1/l}\right)^{(p+1)l} \, dx.
    \end{align*}
Now Bernoulli's inequality, for each positive integer $l$ and $x\in [0,1]$, implies that
    \begin{align*}
     \left(  \frac{1-x}{l} \right)^{l} \leq \left(1-x^{1/l}\right)^l \leq (1-x)^l.
    \end{align*}
Hence, if $n \rightarrow \infty$, then
    \begin{align*}
        \frac{ \int_{\Delta_l} \left| \frac{\partial W_n}{\partial y} (x,y) \right|^p \, dx dy   }{ \int_{\Delta_l} \left|  W_n (x,y) \right|^p \, dx dy   } \sim \frac{ \int_{0}^{1} \left|  P_n^{(\alpha,\alpha)}(x) \right|^p \left(1-x\right)^l \, dx }{ \int_{0}^{1} \left|  P_n^{(\alpha,\alpha)}(x) \right|^p \left(1-x\right)^{(p+1)l} \, dx  }.
    \end{align*}
We may now apply the result of Szeg\"{o} (see \cite{Szego}, Chap. VII) to get
    \begin{align}
      & \int_{0}^{1} \left|  P_n^{(\alpha,\alpha)}(x) \right|^p \left(1-x\right)^l \, dx  \sim n^{\alpha p-2l-2} \quad \text{whenever} \quad 2l<\mu_{\alpha,p}, \label{s1} \\ &\int_{0}^{1} \left|  P_n^{(\alpha,\alpha)}(x) \right|^p \left(1-x\right)^{(p+1)l} \, dx  \sim n^{\alpha p-2(p+1)l-2},  \quad 2(p+1)l<\mu_{\alpha,p}, \label{s2}
      \end{align}
where $\mu_{\alpha,p}=\alpha p -2 + p/2$. If $2(p+1)l<\mu_{\alpha,p}$, we can combine (\ref{s1}) and (\ref{s2}) to see that
    \begin{align*}
     \frac{ \left\| \frac{\partial W_n}{\partial y} \right\|_{L^p(\Delta_l)} } {\left\| W_n\right\|_{L^p(\Delta_l)}} \sim n^{2l}.
    \end{align*}
That is what we wished to prove.
\section*{Acknowledgment}
The author deeply thanks Mirosław Baran and Leokadia Białas-Cież who pointed out some important remarks, corrections and shared their unpublished notes. \\
The author was supported by the Polish National Science Centre (NCN) Opus grant no. 2017/25/B/ST1/00906.






\small{
}

\end{document}